\documentclass[11pt]{article}

\usepackage{amsmath,amssymb,amsthm,accents}
\usepackage{amscd}
\usepackage[initials,short-journals,non-sorted-cites]{amsrefs}
\usepackage{enumitem}
\usepackage{type1cm}
\usepackage{float}
\usepackage[all]{xy}
\usepackage{amsfonts}
\usepackage[dvipdfmx]{graphicx}
\usepackage{comment}
\usepackage{bm}

\usepackage{color}

\allowdisplaybreaks[4]

\newtheorem{theo}{Theorem}[section]
\newtheorem{prop}[theo]{Proposition}
\newtheorem{lemm}[theo]{Lemma}
\newtheorem{cor}[theo]{Corollary}

\newtheorem{Theo}{Theorem}

\theoremstyle{definition}
\newtheorem{defi}[theo]{Definition}
\newtheorem{exam}[theo]{Example}

\newtheorem{remark}[theo]{Remark}

\newcommand{\Z}{\mathbb{Z}}

\newcommand{\R}{\mathbb{R}}

\newcommand{\relmiddle}[1]{\mathrel{}\middle#1\mathrel{}}

\title{Alexander matrices of link quandles associated to quandle homomorphisms and quandle cocycle invariants\footnotemark}
\author{Yuta Taniguchi}
\date{}

\begin{document}

\maketitle
\footnotetext[\ast]{
{\bf keywords:} {quandle, quandle cocycle invariant, twisted Alexander invariant}
}
\section{Introduction}
\label{sect:intro}
The Alexander polynomial \cite{alexander1928topo} is a classical knot invariant which is defined as a generator of the elementary ideal of the Alexader invariant. Using Fox derivative \cite{fox1953free}, we can compute the Alexander polynomial from a matrix, called the Alexander matrix, obtained from a presentation of the knot group. In \cite{lin2001repr}, X. S. Lin generalized the Alexander polynomial to the twisted Alexander polynomial of knots in $S^3$ using regular Seifert surfaces. M. Wada \cite{wada1994twisted} defined an Alexander matrix of a finitely presented group associated to a linear representation and pointed out that Lin's twisted Alexander polynomial can be obtained from the Alexander matrix of the link group associated to a linear representation. 

A quandle which was introduced in \cite{Joy,Mat} is an algebraic structure defined on a set with a binary operation whose definition was motivated from knot theory. D. Joyce and S. V. Matveev associated a quandle to a link, which is called the link quandle or the fundamental quandle. In~\cite{ishiitwisted}, A. Ishii and K. Oshiro introduced a matrix of a finitely presented quandle associated to a quandle homomorphism, which is called an $f$-twisted Alexander matrix. The $f$-twisted Alexander matrix $A(Q,\rho;f_1,f_2)$ is defined when we fix a finite presentation of a quandle $Q$, a quandle homomorphism $\rho$ and an Alexander pair $(f_1,f_2)$. They showed that the $f$-twisted Alexander matrix is an invariant of a pair of a quandle and a quandle homomorphism. Furthermore, they also showed that the Alexander matrix of the link group associated to a linear representation can be recovered from an $f$-twisted Alexander matrix of the link quandle.

The aim of this paper is to construct relationship between the $f$-twisted Alexander matrix and the quandle cocycle invariant, which was introduced in \cite{CJKLS}. The quandle cocycle invariant is defined when we fix a quandle $2$-cocycle. In this paper, given a quandle $2$-cocycle $\theta$, we construct the Alexander pair $(f_{\theta},0)$, which is called the {\it Alexander pair associated with the quandle $2$-cocycle} $\theta$. Furthermore, we show that a certain information of the quandle cocycle invariant using a quandle $2$-cocycle $\theta$ can be obtained from $f$-twisted Alexander matrix using the Alexander pair $(f_{\theta},0)$.
\begin{Theo}[Theorem~\ref{quandle cocycle and ideal}]
Let $D=D_1\cup\cdots\cup D_n$ be a diagram of an $n$-component oriented link $L$. Let $X$ be a quandle and $\rho$ be a quandle homomorphism from the link quandle of $L$ to $X$. Let $c_{\rho}:{\rm Arc}(D)\to X$ be an $X$-coloring of $D$ corresponds to the quandle homomorphism $\rho$. Then, the $0$-th elementary ideal of $A(Q(L),\rho;f_{\theta},0)$ is equal to the ideal generated by $\Pi^n_{i=1}(\Phi_{\theta}(D_i,c_{\rho})-1)\in\Z[A]$.
\end{Theo}

As an application, we distinguish the knot $K_1$ and $K_2$ in the figure~\ref{fig:square granny_1} using an $f$-twisted Alexander matrix.  
\begin{figure}[H]
 	\center
	\includegraphics[height=3cm,width=8cm]{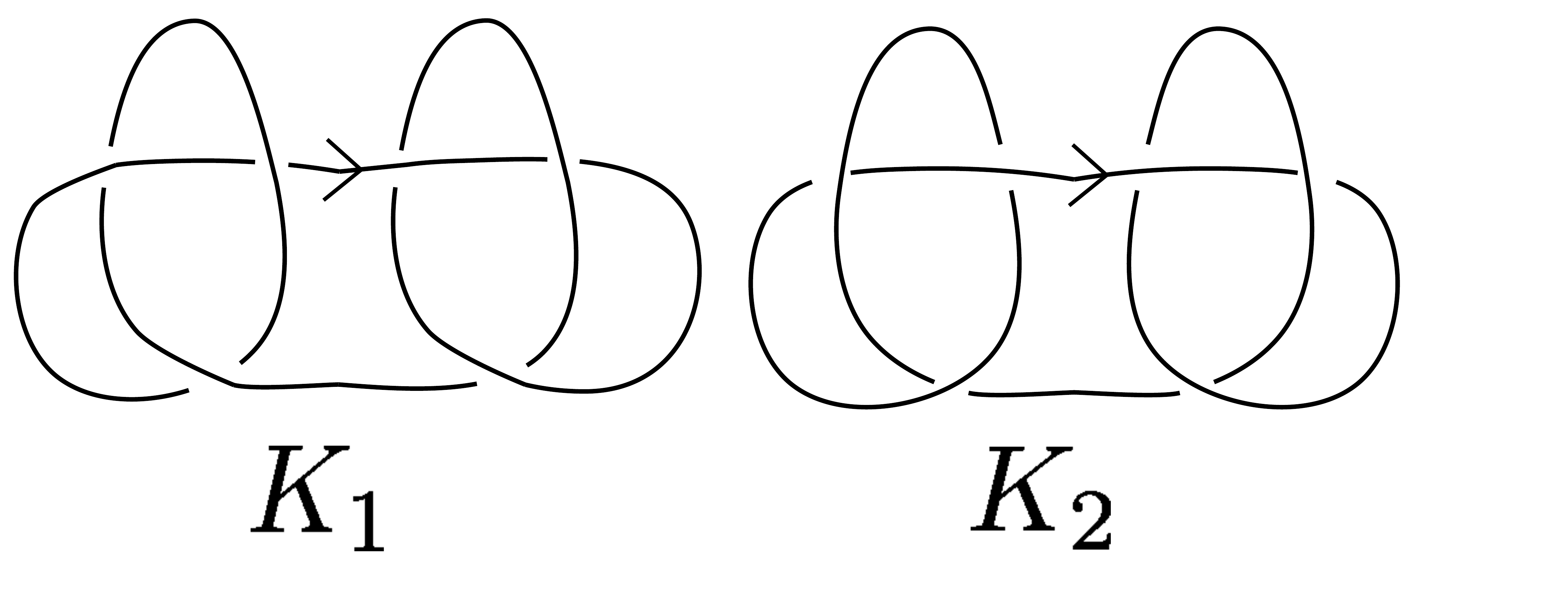}
	\label{fig:square granny_1}
	\caption{Knots whose knot groups are isomorphic.}
 \end{figure}
It is known that the knot group of $K_1$ is isomorphic to the knot group of $K_2$. Since twisted Alexander polynomial depends on the knot group, knots whose knot groups are isomorphic can not be distinguished by using the twisted Alexander polynomial. This implies that an invariant derived from $f$-twisted Alexander matrices is a stronger invariant for oriented knots than the (twisted) Alexander polynomial.

This paper is organized as follows. In Section~\ref{sect:Alexander pair}, we recall that the definitions of a quandle and an Alexander pair. In Section~\ref{sect:quandle presentation}, we review the definition of quandle presentations. In Section~\ref{sect:quandle twisted Alexander}, we review the definition of an $f$-twisted Alexander matrix and see a relationship between the Alexander matrix of link quandles associated to quandle homomorphisms and the Alexander matrix of link groups associated to linear representations, which was introduced by M. Wada. In Section~\ref{sect:cocycle and Alexander matrix}, we prove Theorem~1 and distinguish the knots in Figure~\ref{fig:square granny_1} by using a certain invariant derived from an $f$-twisted Alexander matrix. In Section~\ref{sect:deficiency qdle}, we consider the notion of the deficiency of a quandle and discuss the deficiency of the link quandle of oriented knots.

\section*{Acknowledgements}
The author would like to thank Seiichi Kamada and Hirotaka Akiyoshi for helpful advice and discussions on this research.

\section{Alexander pair}
\label{sect:Alexander pair}

A {\it quandle} is a set $X$ with a binary operation $\ast:X\times X\to X$ satisfying the following three axioms.
\vspace{-0.5\baselineskip}
\begin{itemize}
	\setlength{\itemsep}{0pt}
	\setlength{\parskip}{0pt}
\item[(Q$1$)] For any $x\in X$, we have $x\ast x=x$.
\item[(Q$2$)] For any $y\in X$, the map $\ast y:X\to X$, $x\mapsto x\ast y$ is a bijection.
\item[(Q$3$)] For any $x,y,z\in X$, we have $(x\ast y)\ast z=(x\ast z)\ast(y\ast z)$.
\end{itemize}

\vspace{-0.4\baselineskip}
These axioms correspond to Reidemeister moves in knot theory. By (Q$2$), there exists the binary operation $\ast^{-1}:X\times X\to X$ such that $(x\ast y)\ast^{-1}y=(x\ast^{-1}y)\ast y=x$ for any elements $x,y\in X$. We call $\ast^{-1}$ the {\it dual operation} of $(X,\ast)$.
\begin{exam}
{\rm
Let $G$ be a group. We define an operation on $G$ by $x\ast y=y^{-1}xy$. Then, $G$ is a quandle, which is called the {\it conjugation quandle} of $G$. We denote the conjugation quandle of $G$ by ${\rm Conj}(G)$.
}
\end{exam}
 Let $X$ and $Y$ be quandles. A map $f:X\to Y$ is a {\it quandle homomorphism} if $f(x\ast y)=f(x)\ast f(y)$ for any $x,y\in X$. A quandle homomorphism $f:X\to Y$ is a {\it quandle isomorphism} or a {\it quandle automorphism} if $f$ is a bijection, or if $X=Y$ and it is a bijection, respectively. We denote the set of all quandle homomorphisms from $X$ to $Y$ by ${\rm Hom}(X,Y)$ and the set of all quandle automorphisms of $X$ by ${\rm Aut}(X)$. Then, ${\rm Aut}(X)$ forms a group by $(f\cdot g)(x)=g\circ f(x)$ for $f,g\in{\rm Aut}(X)$ and $x\in X$. Notice that $\ast x:X\to X$ is a quandle automorphism for any $x\in X$. The subgroup of ${\rm Aut}(X)$ generated by the set $\{\ast x\mid x\in X\}$ is called the {\it inner automorphism group} of $X$, which is denoted by ${\rm Inn}(X)$.  We call an element of ${\rm Inn}(X)$ an {\it inner automorphism} of $X$. A quandle $X$ is {\it faithful} if the map $\iota:X\to{\rm Inn}(X)$ defined by $\iota(x)=\ast x$ is an injective map. 

\begin{defi}[\cite{ishiitwisted}]
\label{Alexander pair}
{\rm
Let $X$ be a quandle and $R$ be a ring with the unity $1$.
A pair $(f_1,f_2)$ of maps $f_1,f_2:X\times X\to R$ is an {\it Alexander pair} if $f_1$ and $f_2$ satisfy the following three conditions:
\begin{itemize}
\item For any $x\in X$, we have $f_1(x,x)+f_2(x,x)=1$.
\item For any $x,y\in X$, $f_1(x,y)$ is a unit of $R$.
\item For any $x,y,z\in X$, we have
\begin{align*}
&f_1(x\ast y,z)f_1(x,y)=f_1(x\ast z,y\ast z)f_1(x,z),\\
&f_1(x\ast y,z)f_2(x,y)=f_2(x\ast z,y\ast z)f_1(y,z), {\rm and}\\
&f_2(x\ast y,z)=f_1(x\ast z,y\ast z)f_2(x,z)+f_2(x\ast z,y\ast z)f_2(y,z).
\end{align*}
\end{itemize}
}
\end{defi}
In \cite{andruskiewitsch2003racks}, N. Andruskiewitsch and M. Gra\~{n}a introduced the notion of a quandle module. A. Ishii and K. Oshiro introduced the notion of an Alexander pair $(f_1,f_2)$ as stated above, which corresponds to a quandle module structure $(\eta,\tau)$ in~\cite{andruskiewitsch2003racks}. We adopt Ishii-Oshiro's notion in this paper.
\begin{exam}
{\rm
Let $X$ be a quandle. We define maps $f_1,f_2:X\times X\to \Z[t^{\pm 1}]$, where $\Z[t^{\pm 1}]$ is the ring of Laurent polynomials with integer coefficients, by
\[
f_1(x,y):=t, \quad f_2(x,y):=1-t.
\]

 Then, the pair $(f_1,f_2)$ is an Alexander pair.
}
\end{exam}
\begin{exam}
{\rm
Let $G$ be a group and $R[G]$ be the group ring over a commutative ring $R$. Let $X$ be the conjugation quandle of $G$. We define maps $f_1,f_2:X\times X\to R[G]$ by
\[
f_1(x,y):=y^{-1}, \quad f_2(x,y):=y^{-1}x-y^{-1}.
\] 

 Then, the pair $(f_1,f_2)$ is an Alexander pair.
}
\end{exam}

\begin{prop}[\cite{andruskiewitsch2003racks}]
\label{qdle by Alexander pair}
Let $X$ be a quandle with operation $\ast$. Let $R$ be a ring and $M$ be a left $R$-module. Let $f_1,f_2:X\times X\to R$ be maps. If $(f_1,f_2)$ is an Alexander pair, then $X\times M$ with operation $\triangleleft$ defined by
\[
(x,a)\triangleleft (y,b)=(x\ast y,f_1(x,y)a+f_2(x,y)b).
\]
is a quandle.
\end{prop}
Let $Q$ and $X$ be quandles. Let $R$ be a ring and $\rho:Q\to X$ be a quandle homomorphism. Let $f_1,f_2:X\times X\to R$ be maps. We define maps $f_1^{\rho},f_2^{\rho}:Q\times Q\to R$ by $f_1^{\rho}=f_1\circ(\rho\times\rho), f_2^{\rho}=f_2\circ(\rho\times\rho)$. 
\begin{prop}[\cite{ishiitwisted}]
If $f=(f_1,f_2)$ is an Alexander pair, then $f\circ \rho=(f_1^{\rho},f_2^{\rho})$ is an Alexander pair. 
\end{prop}

\section{Quandle presentation}
\label{sect:quandle presentation}
In this section, we review quandle presentations. Refer to~\cite{fenn1992racks,kamada2017surface} for details. Let $S$ be a non-empty set and $F(S)$ be the free group on $S$. We define an equivalence relation $\sim_q$ on $S\times F(S)$ by
\[
(a,x)\sim_q(b,y)\Leftrightarrow\begin{cases}
a=b\\
x=a^ny\quad\textrm{ for some }n\in\Z.
\end{cases}
\]

We denote the quotient set $S\times F(S)/\sim_q$ by $FQ(S)$. We define an operation $\ast$ on $FQ(S)$ by
\[
[(a,x)]\ast[(b,y)]=[(a,xy^{-1}by)].
\]
Then, $FQ(S)$ is a quandle, which is called the {\it free quandle} on $S$.

Let $R$ be a subset of $FQ(S)\times FQ(S)$. Applying the following operations, we enlarge $R$ to a subset of $FQ(S)\times FQ(S)$.
\begin{itemize}
\item For any $x\in FQ(S)$, add $(x,x)$ to $R$.
\item For any $(x,y)\in R$, add $(y,x)$ to $R$.
\item For any $(x,y),(y,z)\in R$, add $(x,z)$ to $R$.
\item For any $(x,y)\in R$ and $z\in FQ(S)$, add $(x\ast z,y\ast z)$ and $(x\ast^{-1} z,y\ast^{-1} z)$ to $R$.
\item For any $(x,y)\in R$ and $z\in FQ(S)$, add $(z\ast x,z\ast y)$ and $(z\ast^{-1} x,z\ast^{-1} y)$ to $R$.
\end{itemize}

A {\it consequence} of $R$ is an element of an expanded $R$ by a finite sequence of the above operations. We denote the set of all cosequences of $R$ by $\langle\langle R\rangle\rangle_{\rm qdle}$. The quotient on $FQ(S)$ by $\langle\langle R\rangle\rangle_{\rm qdle}$ has a quandle operation inherited from $FQ(S)$. We define $\langle S\mid R\rangle:=FQ(S)/\langle\langle R\rangle\rangle_{\rm qdle}$. Then, we have the natural projection ${\rm pr}_{\rm qdle}:FQ(S)\to \langle S\mid R\rangle$. In this paper, we abbreviate ${\rm pr}_{\rm qdle}(x)$ to $x$ for $x\in FQ(S)$.

 We say that a quandle $X$ has a {\it presentation} $\langle S\mid R\rangle$ if $X$ is isomorphic to the quandle $\langle S\mid R\rangle$. We call an element of $S$ or $R$ a {\it generator} or a {\it relator}, respectively. A presentation $\langle S\mid R\rangle$ is {\it finite} if $S$ and $R$ are finite sets.
 
 R. Fenn and C. Rouke \cite{fenn1992racks} showed the following theorem which corresponds to Tietze's theorem for group presentations.
 \begin{theo}
 Let $X$ be a quandle. Any two finite presentations of $X$ are related by a finite sequence of operations {\rm (T1)} and {\rm (T2)}.
 \begin{itemize}
 \item[\rm (T1)] $\langle S\mid R\rangle\leftrightarrow\langle S\mid R\cup\{ r\}\rangle\ (r\in\langle\langle R\rangle\rangle_{\rm qdle})$ .
 \item [\rm (T2)] $\langle S\mid R\rangle\leftrightarrow\langle S\cup\{ y\}\mid R\cup\{(y,w_y)\}\rangle\ (y\notin FQ(S), w_y\in FQ(S))$.
 \end{itemize}
 \end{theo}
 We call the operations (T1) and (T2) {\it Tietze's moves} on quandle presentations.
\begin{exam}
Let $D$ be a diagram of an oriented link $L$ and ${\rm Arc}(D)$ be the set of arcs of $D$. For each crossing $\chi$, we define a relator $r_{\chi}$ by $(x_i\ast x_j,x_k)$, where $x_i,x_j,x_k$ are the arcs as shown in Figure~\ref{crossing_relation}. 
\begin{figure}[h]
  	\begin{center}
  	\includegraphics[height=3cm,width=4cm]{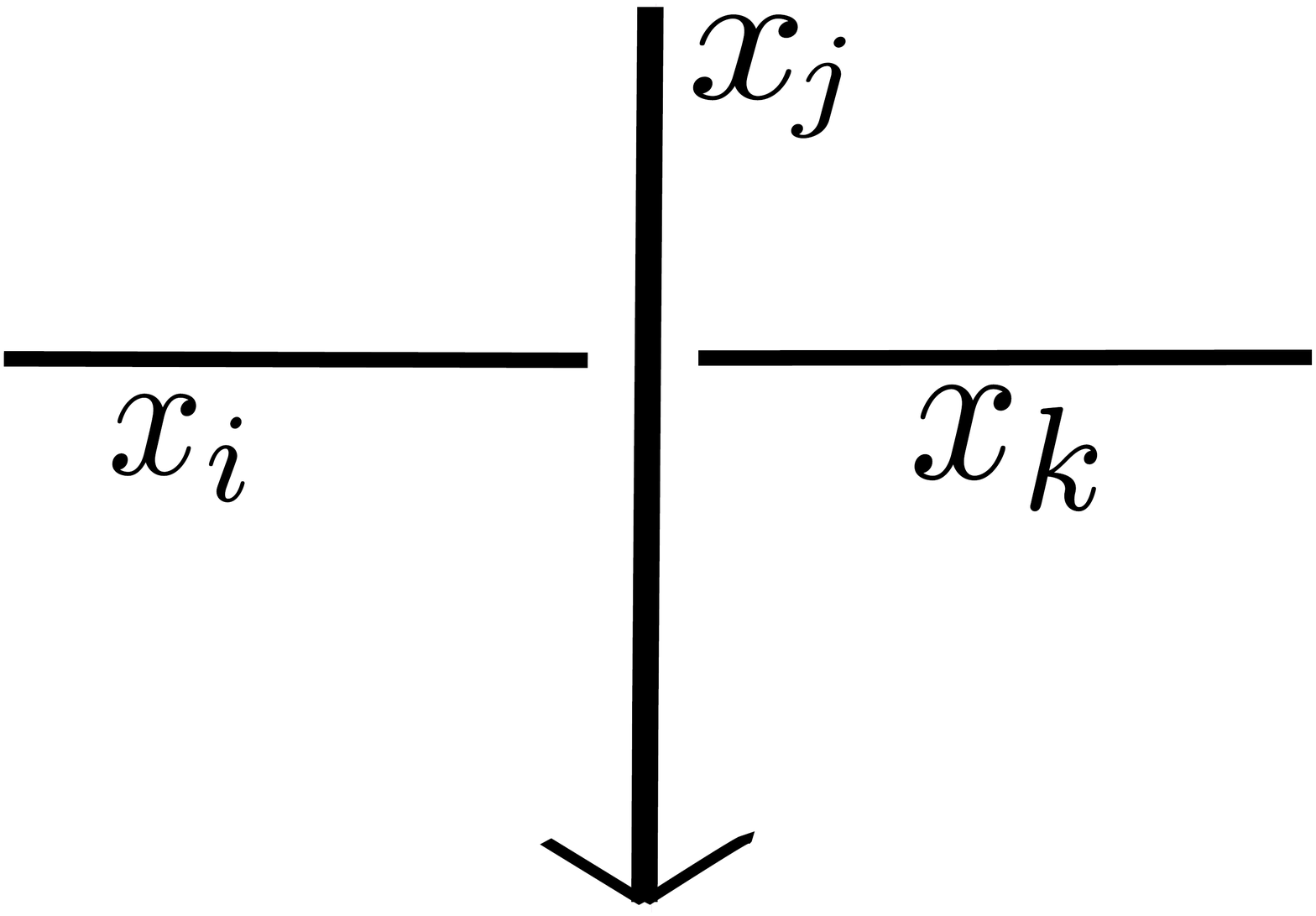}
  	\caption{The arcs around the crossing $\chi$}
  	\label{crossing_relation}
  	\end{center}
 \end{figure}
A quandle which has the following presentation is called the {\it link quandle} or the {\it fundamental quandle} of $L$, which is denoted by $Q(L)$:
\[
\langle {\rm Arc}(D)\mid \{ r_{\chi}\mid\chi:\textrm{ a crossing of }D\} \rangle.
\]
This presentation is called the {\it Wirtinger presentation} of $Q(L)$ with respect to $D$. If $L$ is a knot, then we also call $Q(L)$ the {\it knot quandle}. 
\end{exam}
\section{$f$-twisted Alexander matrices}
\label{sect:quandle twisted Alexander}
In this section, we review the definition of $f$-twisted Alexander matrices. Refer to \cite{ishiitwisted} for details. Let $S=\{ x_1,\ldots,x_n\}$ be a finite set, $Q$ be a quandle equipped a finite presentation $\langle x_1,\ldots,x_n\mid r_1,\ldots,r_m\rangle$, $X$ be a quandle and $\rho:Q\to X$ be a quandle homomorphism.
Let $R$ be a ring and $f=(f_1,f_2)$ be an Alexander pair of maps $f_1,f_2:X\times X\to R$. We recall that the pair $f\circ\rho=(f^{\rho}_1,f^{\rho}_2)$ is also an Alexander pair. For $j\in\{1,\ldots,n\}$, let us define a map $\frac{\partial_{f\circ\rho}}{\partial{x_j}}:FQ(S)\to R$ by the following rules:
\begin{itemize}
\item For any $x,y\in FQ(S)$, we have \\
$\frac{\partial_{f\circ\rho}}{\partial{x_j}}(x\ast y)=f^{\rho}_1(x,y)\frac{\partial_{f\circ\rho}}{\partial{x_j}}(x)+f^{\rho}_2(x,y)\frac{\partial_{f\circ\rho}}{\partial{x_j}}(y)$.
\item For any $x,y\in FQ(S)$, we have\\
 $\frac{\partial_{f\circ\rho}}{\partial{x_j}}(x\ast^{-1} y)=f^{\rho}_1(x\ast^{-1}y,y)^{-1}\frac{\partial_{f\circ\rho}}{\partial{x_j}}(x)-f^{\rho}_1(x\ast^{-1}y,y)^{-1}f^{\rho}_2(x\ast^{-1}y,y)\frac{\partial_{f\circ\rho}}{\partial{x_j}}(y)$.
\item For each $i\in\{1,\ldots,n\}$, we have $\frac{\partial_{f\circ\rho}}{\partial{x_j}}(x_i)=\delta_{ij}$, where $\delta_{ij}$ is the Kronecker delta.
\end{itemize}

The map $\frac{\partial_{f\circ\rho}}{\partial{x_j}}:FQ(S)\to R$ is called the {\it $f\circ\rho$-derivative with respect to $x_j$}.
 For a relator $r=(r_1,r_2)$, we define $\frac{\partial_{f\circ\rho}}{\partial{x_j}}(r):=\frac{\partial_{f\circ\rho}}{\partial{x_j}}(r_1)-\frac{\partial_{f\circ\rho}}{\partial{x_j}}(r_2)$.

Let us consider the $m\times n$ matrix $A(Q,\rho;f_1,f_2)$, which is defined by
\[
A(Q,\rho;f_1,f_2)=\left(
\begin{array}{ccc}
\frac{\partial_{f\circ\rho}}{\partial{x_1}}(r_1) & \cdots & \frac{\partial_{f\circ\rho}}{\partial{x_n}}(r_1) \\
\vdots & \ddots & \vdots \\
\frac{\partial_{f\circ\rho}}{\partial{x_1}}(r_m) & \cdots & \frac{\partial_{f\circ\rho}}{\partial{x_n}}(r_m)
\end{array}
\right).
\]
We call this matrix $A(Q,\rho;f_1,f_2)$ {\it Alexander matrix of the finite presentation $\langle x_1,\ldots,x_n\mid r_1,\ldots,r_m\rangle$ associated to the quandle homomorphism $\rho$} or the {\it $f$-twisted Alexander matrix} of $(Q,\rho)$ with respect to the quandle presentation $\langle x_1,\ldots,x_n\mid r_1,\ldots,r_m\rangle$ \cite{ishiitwisted}. 

Suppose that $R$ is a commutative ring. The {\it $d$-th elementary ideal} of $A(Q,\rho;f_1,f_2)$, which is denoted by $E_d(A(Q,\rho;f_1,f_2))$, is the ideal generated by all $(n-d)$-minors of $A(Q,\rho;f_1,f_2)$ if $n-m\leq d<n$, and
\[
E_d(A(Q,\rho;f_1,f_2))=\begin{cases}
0\quad\ \textrm{if }d<n-m,\\
R\quad\textrm{if }n\leq d.
\end{cases}
\] 

Let $Q^{\prime}$  be a quandle with a finite presentation $\langle x^{\prime}_1,\ldots,x^{\prime}_{n^{\prime}}\mid r^{\prime}_1,\ldots,r^{\prime}_{m^{\prime}}\rangle$ and $\rho^{\prime}:Q^{\prime}\to X$ be a quandle homomorphism. In \cite{ishiitwisted}, A. Ishii and K. Oshiro showed that if there exists a quandle isomorphism $\varphi:Q\to Q^{\prime}$ such that $\rho=\rho^{\prime}\circ\varphi$, then $A(Q,\rho;f_1,f_2)$ and $A(Q^{\prime},\rho^{\prime};f_1,f_2)$ are related by a finite sequence of the following transformations $({\rm M1})\sim({\rm M4})$:
\begin{align*}
&({\rm M1})\ ({\bm a_1},\ldots,{\bm a_i},\ldots,{\bm a_j},\ldots,{\bm a_n})\leftrightarrow({\bm a_1},\ldots,{\bm a_i}+{\bm a_j}r,\ldots,{\bm a_j},\ldots,{\bm a_n})\ (r\in R),\\
&({\rm M2})\ \left(
\begin{array}{c}
{\bm a_1} \\
\vdots \\
{\bm a_i}\\
\vdots \\
{\bm a_j} \\
\vdots \\
{\bm a_n}
\end{array}
\right)\leftrightarrow\left(
\begin{array}{c}
{\bm a_1} \\
\vdots \\
{\bm a_i}+r{\bm a_j} \\
\vdots \\
{\bm a_j} \\
\vdots \\
{\bm a_n}
\end{array}
\right)\ (r\in R),\ \ 
({\rm M3})\ A\leftrightarrow\left(
\begin{array}{c}
A\\
{\bm 0}
\end{array}
\right),\ \ 
({\rm M4})\ A\leftrightarrow\left(
\begin{array}{cc}
A & {\bm 0}\\
{\bm 0} & 1
\end{array}
\right).
\end{align*}
Furthermore, we have
\[
E_d(A(Q,\rho;f_1,f_2))=E_d(A(Q^{\prime},\rho^{\prime};f_1,f_2))
\]
 if $R$ is a commutative ring.
 \begin{exam}
 {\rm
 Let $X=\{x_1,x_2,x_3,x_4\}$ be the vertices of a regular tetrahedron as in Figure~\ref{tetrahedron}. We define a binary operation $\ast$ on $X$ as follows. Look at the bottom face from the vertex $x_j$. When we rotate the tetrahedron by $120^\circ$ counterclockwise, the vertex $x_i$ is moved to $x_i\ast x_j$. For example, the element $x_2\ast x_1$ is $x_3$. Then, $X=(X,\ast)$ is a quandle.

 \begin{figure}[h]
 	\center
	\includegraphics[height=3.5cm,width=7.4cm]{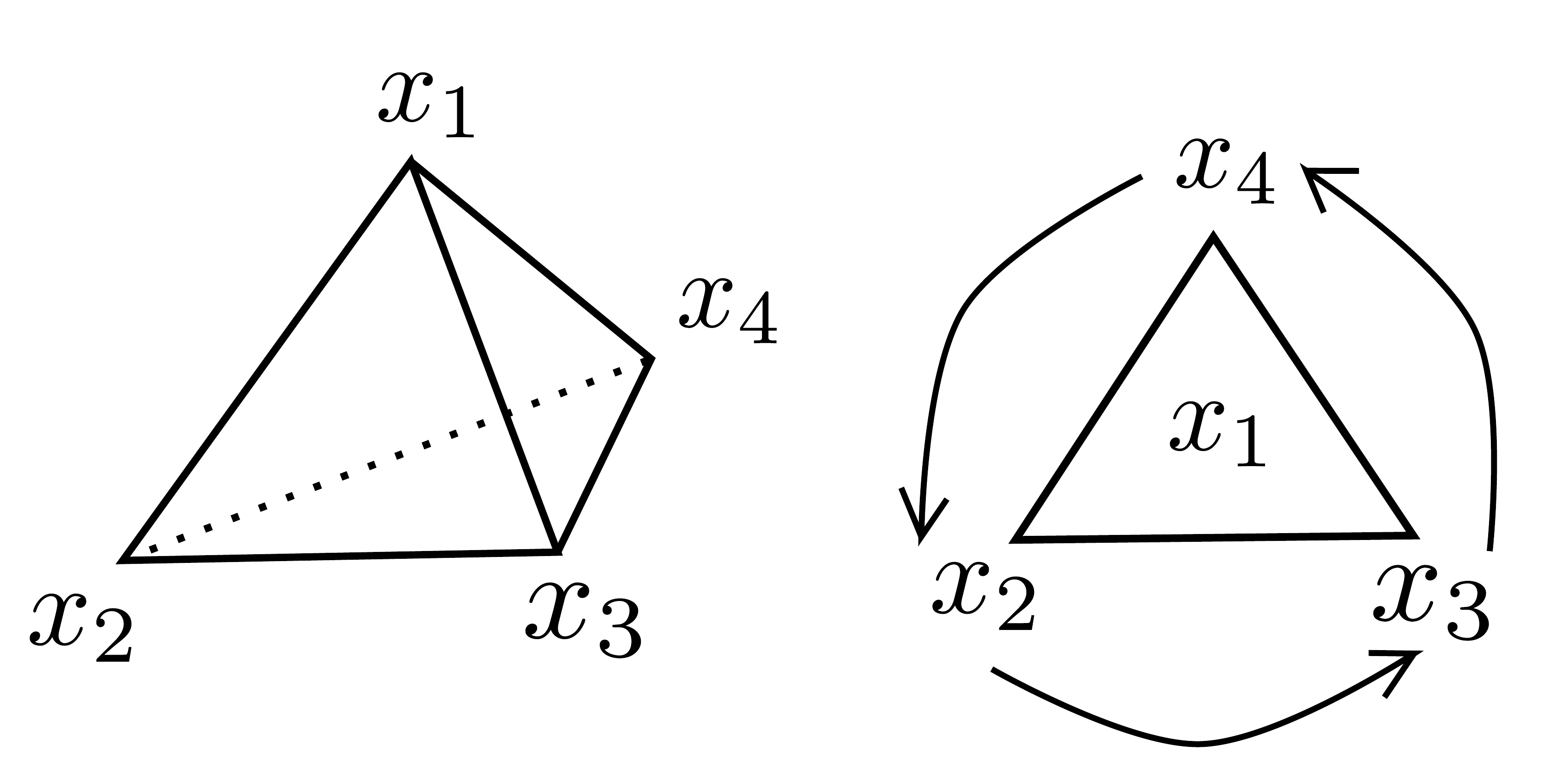}
	\caption{The quandle $X$}
	\label{tetrahedron}
\end{figure}
The multiplication table of $X$ is given as follows:
\begin{table}[H]
\begin{center}
\begin{tabular}{c|cccc}
$\ast$ & $x_1$ & $x_2$ & $x_3$ & $x_4$ \\ \hline
$x_1$ & $x_1$ & $x_4$ & $x_2$ & $x_3$ \\
$x_2$ & $x_3$ & $x_2$ & $x_4$ & $x_1$\\
$x_3$ & $x_4$ & $x_1$ & $x_3$ & $x_2$\\
$x_4$ & $x_2$ & $x_3$ & $x_1$ & $x_4$
\end{tabular}
\end{center}
\end{table}
\vspace{-2mm}
Thus, $X$ has the following presentation:
\[
\left\langle
x_1,x_2,x_3,x_4
\relmiddle|
\begin{array}{c}
 (x_1\ast x_1,x_1), (x_1\ast x_2,x_4), (x_1\ast x_3,x_2), (x_1\ast x_4,x_3)\\
 (x_2\ast x_1,x_3), (x_2\ast x_2,x_2), (x_2\ast x_3,x_4), (x_2\ast x_4,x_1)\\
 (x_3\ast x_1,x_4), (x_3\ast x_2,x_1), (x_3\ast x_3,x_3), (x_3\ast x_4,x_2)\\
 (x_4\ast x_1,x_2), (x_4\ast x_2,x_3), (x_4\ast x_3,x_1), (x_4\ast x_4,x_4)
\end{array}
\right\rangle
\]

 Using Tietze's moves, we can see that $X$ has the following presentation: 
 \[
 \langle x_1,x_2\mid ((x_1\ast x_2)\ast x_1,x_2),((x_2\ast x_1)\ast x_2,x_1),((x_2\ast x_1)\ast x_1,x_1\ast x_2),((x_1\ast x_2)\ast x_2,x_2\ast x_1)\rangle.
 \]

 We define two maps $f_1,f_2:X\times X\to\Z[t^{\pm 1}]$ by 
 \[
 f_1(x,y)=t,\quad f_2(x,y)=1-t.
 \]

 Then, $f=(f_1,f_2)$ is an Alexander pair. Let $id_X:X\to X$ be the identity map on $X$. We simply denote $f\circ id_X$ by $f$.  
 For each $i\in \{ 1,2\}$, we have
 \begin{eqnarray*}
 &&\frac{\partial_f}{\partial x_i}(((x_1\ast x_2)\ast x_1,x_2)) \\
  &=& \frac{\partial_f}{\partial x_i}((x_1\ast x_2)\ast x_1)-\frac{\partial_f}{\partial x_i}(x_2)\\
 &=& f_1(x_1\ast x_2,x_2)\frac{\partial_f}{\partial x_i}(x_1\ast x_2)+f_2(x_1\ast x_2,x_2)\frac{\partial_f}{\partial x_i}(x_1)-\frac{\partial_f}{\partial x_i}(x_2)\\
 &=&t\frac{\partial_f}{\partial x_i}(x_1\ast x_2)+(1-t)\frac{\partial_f}{\partial x_i}(x_1)-\frac{\partial_f}{\partial x_i}(x_2)\\
 &=&t^2\frac{\partial_f}{\partial x_i}(x_1)+(t-t^2)\frac{\partial_f}{\partial x_i}(x_2)+(1-t)\frac{\partial_f}{\partial x_i}(x_1)-\frac{\partial_f}{\partial x_i}(x_2)\\
 &=&(t^2-t+1)\frac{\partial_f}{\partial x_i}(x_1)+(-t^2+t-1)\frac{\partial_f}{\partial x_i}(x_2).
 \end{eqnarray*}
 Similarly, we have 
 \begin{eqnarray*}
 \frac{\partial_f}{\partial x_i}(((x_2\ast x_1)\ast x_2,x_1))&=&(-t^2+t-1)\frac{\partial_f}{\partial x_i}(x_1)+(t^2-t+1)\frac{\partial_f}{\partial x_i}(x_2),\\
 \frac{\partial_f}{\partial x_i}(((x_2\ast x_1)\ast x_1,x_1\ast x_2))&=&(-t^2-t+1)\frac{\partial_f}{\partial x_i}(x_1)+(t^2+t-1)\frac{\partial_f}{\partial x_i}(x_2),\textrm{ and }\\
  \frac{\partial_f}{\partial x_i}(((x_1\ast x_2)\ast x_2,x_2\ast x_1))&=&(t^2+t-1)\frac{\partial_f}{\partial x_i}(x_1)+(-t^2-t+1)\frac{\partial_f}{\partial x_i}(x_2).
 \end{eqnarray*}
 Hence, we have
 \[
 A(X,id_X;f_1,f_2)=\left(\begin{array}{cc}
t^2-t+1 & -t^2+t-1 \\
-t^2+t-1 & t^2-t+1 \\
-t^2-t+1 & t^2+t-1 \\
t^2+t-1 & -t^2-t+1
\end{array}
\right).
 \]
 and
 \[
  E_d(A(X,id_X;f_1,f_2))=\begin{cases}
  (0) \quad\quad\quad\quad\quad\quad\quad\quad\quad\quad \textrm{if }d\leq0,\\
  (t^2-t+1,-t^2-t+1)\quad \textrm{if }d=1,\\
  \Z[t^{\pm 1}]\quad\quad\quad\quad\quad\quad\quad\quad\ \ \textrm{if } d\geq 2.
  \end{cases}
 \]
 }
 \end{exam}
 \begin{remark}
 {\rm
 Let $D$ be a diagram of an oriented knot $K$ and $R$ be a unique factorization domain. Let $k$ be a positive integer, and we put $G:=GL(k;R)$. Let
 \begin{align}
 & Q(K)=\langle x_1,\ldots,x_n\mid (u_1\ast v_1,w_1),\ldots,(u_m\ast v_m,w_n)\rangle \tag{4.1}\\
 & G(K)=\langle x_1,\ldots,x_n\mid v_1^{-1}u_1v_1w_1^{-1},\ldots,v_m^{-1}u_mv_mw_m^{-1}\rangle_{\rm grp} \tag{4.2}
 \end{align}
be the Wirtinger presentations of the knot quandle $Q(K)$ and the knot group $G(K)=\pi_1(S^3\backslash K)$ with respect to the diagram $D$. Let $\rho_{\rm grp}:G(K)\to G$ be a group homomorphism. Then, $\rho_{\rm grp}$ induces the quandle homomorphism $\rho_{\rm qdle}:Q(K)\to{\rm Conj}(G)$. Let $F(S)$ be the free group on the finite set $S=\{x_1,\ldots, x_n\}$. The Fox derivative \cite{fox1953free} with respect to $x_i$ is the $R$-module homomorphism $\frac{\partial}{\partial x_i}:R[F(S)]\to R[F(S)]$ which satisfies the following conditions:
\begin{itemize}
\item For any $p,q\in F(S)$, $\frac{\partial}{\partial x_i}(pq)=\frac{\partial}{\partial x_i}(p)+p\frac{\partial}{\partial x_i}(q)$.
\item For any $1\leq j\leq n$, $\frac{\partial}{\partial x_i}(x_j)=\delta_{ij}$.
\end{itemize}

Let ${\rm pr_{grp}}:F(S)\to G(K)$ be the canonical projection and $\alpha:G(K)\to\Z=\langle t\rangle_{\rm grp}$ be the abelianization which sends the meridian of $K$ to $t^{-1}$. We denote the linear extensions of ${\rm pr_{grp}}$, $\alpha$ and $\rho_{\rm grp}$ by the same symbol ${\rm pr_{grp}}:R[F(S)]\to R[G(K)]$, $\alpha:R[G(K)]\to R[t^{\pm 1}]$ and $\rho_{\rm grp}:R[G(K)]\to M(k,k;R)$, where $M(k,k;R)$ is the matrix algebra of degree $k$ over $R$.

Then, the {\it Alexander matrix of the presentation $(4.2)$ associated to the representation $\rho_{\rm grp}$}, which was introduced by M. Wada \cite{wada1994twisted}, is
\[
\left(
((\rho_{\rm grp}\otimes\alpha)\circ{\rm pr_{grp}})\left(\frac{\partial}{\partial x_j}(v_i^{-1}u_iv_iw_i^{-1})\right)
\right).
\]
If $k=1$ and $\rho_{\rm grp}$ is a trivial representation, this matrix is simply called the {\it Alexander matrix of the presentation $(4.2)$}.  

 We define the maps $f_1,f_2:{\rm Conj}(G)^2\to M(k,k;R[t^{\pm 1}])$ by 
 \[
 f_1(x,y):=y^{-1}t,\quad f_2(x,y):=y^{-1}x-y^{-1}t.
 \] 

 We see that the pair $f=(f_1,f_2)$ is an Alexander pair. Then, the $f$-twisted Alexander matrix $A(Q(K),\rho_{\rm qdle};f_1,f_2)$ with respect to the presentation $(4.1)$ coincides with the Alexander matrix of the presentation $(4.2)$ associated to the $\rho_{\rm grp}$ (see \cite{ishiitwisted}).

 }
 \end{remark}


\section{Quandle cocycle invariants and $f$-twisted Alexander matrices}
\label{sect:cocycle and Alexander matrix}
 
 In this section, we explain a relationship between quandle cocycle invariants and $f$-twisted Alexander matrices. At first, we review the definition of the quandle cocycle invariant.

Let $X$ be a quandle and $A$ be an abelian group. Assume that the operation of $A$ is written multiplicatively. A map $\theta: X\times X\to A$ is a {\it quandle $2$-cocycle} if $\theta$ satisfies the following conditions (cf \cite{CJKLS}).
\begin{itemize}
\item For any $x\in X$, we have $\theta(x,x)=e$, where $e$ is the identity element of $A$.
\item For any $x,y,z\in X$, we have $\theta(x\ast y,z)\theta(x,y)=\theta(x\ast z,y\ast z)\theta(x,z)$.
\end{itemize}

Let $D=D_1\cup\cdots\cup D_n$ be a diagram of an $n$-component oriented link $L$. A map $c:{\rm Arc}(D)\to X$ is an {\it $X$-coloring} of $D$ if $c$ satisfies the following condition at every crossing of $D$.
\begin{itemize}
\item Let $x_i,x_j,x_k$ be the arcs around a crossing as shown in Figure~\ref{crossing_relation}. Then, $c(x_i)\ast c(x_j)=c(x_k)$. 
\end{itemize}

Let $c:{\rm Arc}(D)\to X$ be an $X$-coloring of $D$. For each crossing $\chi$, we define an element of $A$ denoted by $\Phi_{\theta}(\chi,c)$ as follows. Let $x_i,x_j,x_k$ be the arcs around $\chi$ as shown in Figure~\ref{crossing_relation}. Then, we define $\Phi_{\theta}(\chi,c)$ by
\[
\Phi_{\theta}(\chi,c)=\begin{cases}
\theta(c(x_i),c(x_j))\ \ \ \textrm{if }\chi\textrm{ is positive}\\
\theta(c(x_i),c(x_j))^{-1}\ \textrm{if }\chi\textrm{ is negative}
\end{cases}
\]
and call it the {\it weight}. We denote the set of crossings at which the under-arcs belong to the component $D_i$ by $UC_i$. Then, the {\it component-wise quandle cocycle invariant} $\Phi_{\theta}(L)$ is the multiset
\[
\Phi_{\theta}(L)=\{(\Phi_{\theta}(D_1,c),\ldots,\Phi_{\theta}(D_n,c))\mid c:\textrm{an }X\textrm{-coloring of }D\},
\]
 where
\[
\Phi_{\theta}(D_i,c):=\Pi_{\chi\in UC_i}\Phi_{\theta}(\chi,c).
\]

Let $\eta:A\to\Z[A]$ be the natural inclusion of $A$ into the group ring $\Z[A]$. We denote the composite map $\eta\circ\theta:X\times X\to \Z[A]$ by $f_\theta$. Let $0:X\times X\to \Z[A]$ be the constant map onto $0$. 

\begin{lemm}
The pair $f_{\theta}=(f_{\theta},0)$ is an Alexander pair. In this paper, the Alexander pair $f_{\theta}$ is called the Alexander pair associated with the quandle $2$-cocycle $\theta$.
\end{lemm}
\begin{proof}
We can verify this lemma by direct calculation.
\end{proof}

The following theorem is the one of the main results of this paper.
\begin{theo}
\label{quandle cocycle and ideal}
Let $L=L_1\cup\cdots\cup L_n$ be an $n$-component oriented link and $D=D_1\cup\cdots\cup D_n$ be a diagram of $L$. Let $X$ be a quandle, $A$ be an abelian group, $\rho:Q(L)\to X$ be a quandle homomorphism and $\theta:X\times X\to A$ be a quandle $2$-cocycle.  Let $c_{\rho}:{\rm Arc}(D)\to X$ be an $X$-coloring of $D$ corresponds to the quandle homomorphism $\rho$. Then, we have
\[
E_0(A(Q(L),\rho;f_\theta,0))=(\Pi^n_{i=1}(\Phi_{\theta}(D_i,c_\rho)-1))\subset\Z[A],
\]
 where $1$ is the multiplicative identity element of $\Z[A]$.
\end{theo}
\begin{proof}
For each $i\in\{1,\ldots,n\}$, we denote the cardinality of $UC_i$ by $m_i$.
\begin{figure}[H]
 	\center
	\includegraphics[height=3.2cm,width=11cm]{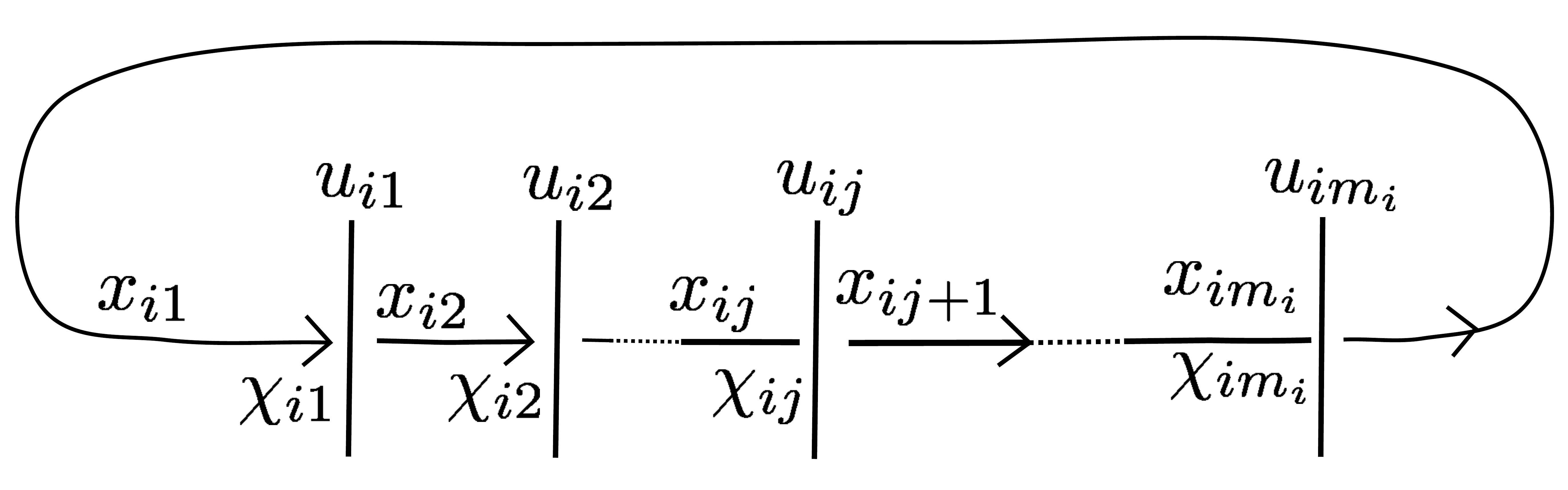}
	\caption{An illustration of the $i$-th component}
	\label{linkpresentation}
\end{figure}
 As illustrated in Figure~\ref{linkpresentation}, consider the arcs $x_{i1},\ldots,x_{im_i}$ along the orientation of the $i$-th component. Let $u_{ij}$ be the arc which divides $x_{ij}$ and $x_{ij+1}$ and $\chi_{ij}$ be the crossing among $x_{ij}$, $u_{ij}$ and $x_{ij+1}$.

For the crossing $\chi_{ij}$, we define the relator $r_{ij}$ by $(x_{ij}\ast^{\varepsilon_{ij}}u_{ij},x_{ij+1})$, where $\varepsilon_{ij}$ is the sign of the crossing $\chi_{ij}$.
Then, $\langle x_{11},x_{12},\ldots,x_{1m_1},x_{21},\ldots, x_{nm_n}\mid  r_{11},r_{12},\ldots,r_{1m_1},r_{21},\ldots, r_{nm_n}\rangle$ is a presentation of $Q(L)$. Hence, we may assume that $Q(L)=\langle x_{11},x_{12},\ldots,x_{1m_1},x_{21},\ldots, x_{nm_n}\mid  r_{11},r_{12},\ldots,r_{1m_1},r_{21},\ldots, r_{nm_n}\rangle$.

Let $\rho:Q(L)\to X$ be a quandle homomorphism. We define a map $c_{\rho}:{\rm Arc}(D)=\{ x_{ij}\}\to X$ by $c_{\rho}(x_{ij}):=\rho(x_{ij})$. Then, the map $c_{\rho}$ is an $X$-coloring of $D$. 

If the sign of $\chi_{ij}$ is positive, we have
\begin{eqnarray*}
\frac{\partial_{f_{\theta}\circ\rho}}{\partial{x_k}}(r_{ij})&=&\frac{\partial_{f_{\theta}\circ\rho}}{\partial{x_k}}(x_{ij}\ast u_{ij})-\frac{\partial_{f_{\theta}\circ\rho}}{\partial{x_k}}(x_{ij+1})\\
&=&f_{\theta}(\rho(x_{ij}),\rho(u_{ij}))\frac{\partial_{f_{\theta}\circ\rho}}{\partial{x_k}}(x_{ij})+0(\rho(x_{ij}),\rho(u_{ij}))\frac{\partial_{f_{\theta}\circ\rho}}{\partial{x_k}}(x_{ij})-\frac{\partial_{f_{\theta}\circ\rho}}{\partial{x_k}}(x_{ij+1})\\
&=&\theta(c_{\rho}(x_{ij}),c_{\rho}(u_{ij}))\frac{\partial_{f_{\theta}\circ\rho}}{\partial{x_k}}(x_{ij})-\frac{\partial_{f_{\theta}\circ\rho}}{\partial{x_k}}(x_{ij+1})\\
&=&\Phi_{\theta}(\chi_{ij},c_{\rho})\frac{\partial_{f_{\theta}\circ\rho}}{\partial{x_k}}(x_{ij})-\frac{\partial_{f_{\theta}\circ\rho}}{\partial{x_k}}(x_{ij+1}).
\end{eqnarray*}

If the sign of $\chi_{ij}$ is negative, we have
\begin{eqnarray*}
\frac{\partial_{f_{\theta}\circ\rho}}{\partial{x_k}}(r_{ij})&=&\frac{\partial_{f_{\theta}\circ\rho}}{\partial{x_k}}(x_{ij}\ast^{-1} u_{ij})-\frac{\partial_{f_{\theta}\circ\rho}}{\partial{x_k}}(x_{ij+1})\\
&=&f_{\theta}(\rho(x_{ij})\ast^{-1}\rho(u_{ij}),\rho(u_{ij}))^{-1}\frac{\partial_{f_{\theta}\circ\rho}}{\partial{x_k}}(x_{ij})\\
&&-f_{\theta}(\rho(x_{ij})\ast^{-1}\rho(u_{ij}),\rho(u_{ij}))^{-1}0(\rho(x_{ij}),\rho(u_{ij}))\frac{\partial_{f_{\theta}\circ\rho}}{\partial{x_k}}(x_{ij})-\frac{\partial_{f_{\theta}\circ\rho}}{\partial{x_k}}(x_{ij+1})\\
&=&\theta(c_{\rho}(x_{ij})\ast^{-1}c_{\rho}(u_{ij}),c_{\rho}(u_{ij}))^{-1}\frac{\partial_{f_{\theta}\circ\rho}}{\partial{x_k}}(x_{ij})-\frac{\partial_{f_{\theta}\circ\rho}}{\partial{x_k}}(x_{ij+1})\\
&=&\theta(c_{\rho}(x_{ij+1}),c_{\rho}(u_{ij}))^{-1}\frac{\partial_{f_{\theta}\circ\rho}}{\partial{x_k}}(x_{ij})-\frac{\partial_{f_{\theta}\circ\rho}}{\partial{x_k}}(x_{ij+1})\\
&=&\Phi_{\theta}(\chi_{ij},c_{\rho})\frac{\partial_{f_{\theta}\circ\rho}}{\partial{x_k}}(x_{ij})-\frac{\partial_{f_{\theta}\circ\rho}}{\partial{x_k}}(x_{ij+1}).
\end{eqnarray*}

Thus, it holds that
\[
A(Q(L),\rho;f_{\theta},0)=
\left(
\begin{array}{ccc}
A_1 &   & O \\
 & \ddots  & \\
 O &  & A_n
\end{array}
\right),
\]
 where 
\[
A_i=
\left(
\begin{array}{ccccc}
\Phi_{\theta}(\chi_{i1},c_\rho) & 1 &  &  &  \\
 & \Phi_{\theta}(\chi_{i2},c_\rho) & 1 &  & \\
 &  & \ddots & \ddots & \\
 & & & \Phi_{\theta}(\chi_{im_{i}-1},c_{\rho}) & 1\\
 1 &   &   &   & \Phi_{\theta}(\chi_{im_i},c_\rho)
\end{array}
\right).
\]

We remark that $A_{i}$ is a square matrix of order $m_i$. In the direct calculation, we have
\[
{\rm det}(A_i)=\Phi_{\theta}(D_i,c_{\rho})-1.
\]

Hence, it holds that
\begin{eqnarray*}
{\rm det}(A(Q(L),\rho;f_{\theta},0))&=&\Pi^n_{i=1}{\rm det}(A_i)\\
&=&\Pi^n_{i=1}(\Phi_{\theta}(D_i,c_{\rho})-1).
\end{eqnarray*}

Since $A(Q(L),\rho;f_{\theta},0)$ is a square matrix, $E_{0}(A(Q(L),\rho;f_{\theta},0))$ is the ideal generated by ${\rm det}(A(Q(L),\rho;f_{\theta},0))$, which implies the assertion.
\end{proof}
Using Theorem~\ref{quandle cocycle and ideal},  we have the following corollary.
\begin{cor}
\label{cor:quandle cocycle knot}
Let $D$ be a diagram of an oriented knot $K$, $X$ be a finite quandle, $A$ be an abelian group and $\theta:X\times X\to A$ be a quandle $2$-cocycle. Then, the multiset $\{E_0(A(Q(K),\rho;f_{\theta},0))\mid\rho\in{\rm Hom}(Q(K),X)\}$ is equal to the multiset $\{(\Phi_{\theta}(D,c)-1)\subset\Z[A]\mid c:{\rm an}\ X{\rm -coloring}\}$.
\end{cor}
\begin{proof}
The map $\Psi:{\rm Hom}(Q(K),X)\to{\rm Col}_X(D)$ which is defined by $\Psi(\rho):=c_{\rho}$ is a bijection (see Proposition~8.9.4 of \cite{kamada2017surface}). Since $X$ is a finite, the cardinalities of ${\rm Hom}(Q(K),X)$ and ${\rm Col}_X(D)$ are finite. Thus, we have
\begin{eqnarray*}
\{\Phi_{\theta}(D,c)-1\mid c:{\rm an}\ X{\rm -coloring}\}
&=&\{\Phi_{\theta}(D,c_{\rho})-1\mid \rho\in{\rm Hom}(Q(K),X)\}\\
&=&\{E_0(A(Q(K),\rho;f_{\theta},0))\mid\rho\in{\rm Hom}(Q(K),X)\}.
\end{eqnarray*}
\end{proof}

\begin{exam}
\label{ex:square granny}
We consider a {\it granny knot} $K_1$ and a {\it square knot} $K_2$ (see Figure~\ref{fig:square granny}).
\begin{figure}[H]
 	\center
	\includegraphics[height=3cm,width=8cm]{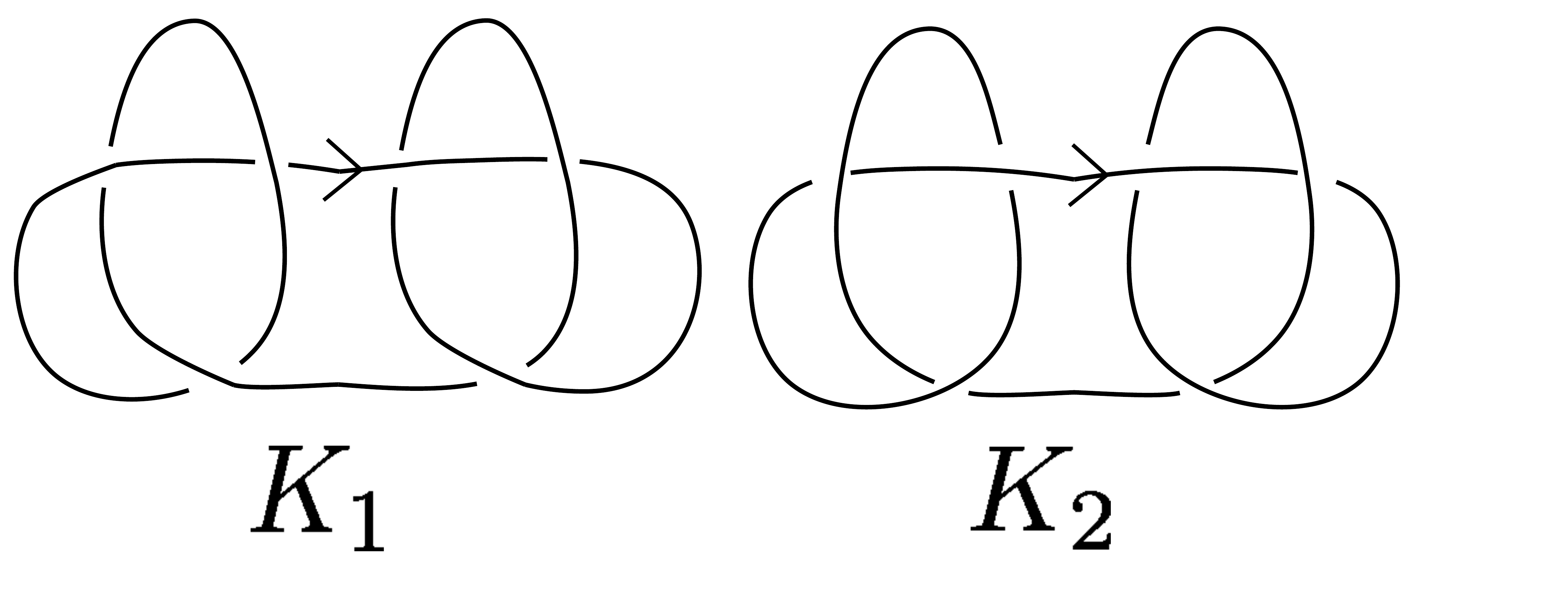}
	\caption{A granny knot $K_1$ and a square knot $K_2$}
	\label{fig:square granny}
 \end{figure}
 We remark that the knot group $G(K_1)$ is isomorphic to $G(K_2)$. We demonstrate that these knots are not equivalent by using $0$-th elementary ideal of $f$-twisted Alexander matrices.

Let $X$ be the conjugacy class of the permutation group of order $4$ consisting of cyclic elements of length four. Then, $X$ becomes a quandle under the conjugation. We note that the quandle $X$ is faithful. By Appendix of \cite{carter2006ribbon}, there exists a quandle $2$-cocycle $\theta:X\times X\to \Z_4=\langle u\mid u^4\rangle_{\rm grp}$ such that $\Phi_{\theta}(3_1)=\{\underbrace{e,\ldots,e}_{6},\underbrace{u,\ldots,u}_{24}\}$, where $3_1$ is the right-hand trefoil. By Theorem $9.1$ of \cite{carter2001computations}, the quandle cocycle invariant of the left-hand trefoil, which is denoted by $3_1!$, is $\{\underbrace{e,\ldots,e}_{6},\underbrace{u^{-1},\ldots,u^{-1}}_{24}\}$. Since $K_1=3_1\# 3_1$ and $K_2=3_1\# 3_1!$, the quandle cocycle invariants of these knots are the following multisets (see Proposition $5.1$ of \cite{nosaka2011homotopy}):
\begin{align*}
& \Phi_{\theta}(K_1)=\{\underbrace{e,\ldots,e}_{6},\underbrace{u,\ldots,u}_{48},\underbrace{u^2,\ldots,u^2}_{96}\},\\
& \Phi_{\theta}(K_2)=\{\underbrace{e,\ldots,e}_{102},\underbrace{u,\ldots,u}_{48}\}.
\end{align*}
Thus, by using Corollary \ref{cor:quandle cocycle knot}, we have
\begin{align*}
& \{ E_0(A(Q(K_1),\rho;f_{\theta},0))\mid\rho\in{\rm Hom}(Q(K_1),X)\}\\
& =\{\underbrace{(0),\ldots,(0)}_{6},\underbrace{(u-1),\ldots,(u-1)}_{48},\underbrace{(u^2-1),\ldots,(u^2-1)}_{96}\}\textrm{ and}\\
& \{ E_0(A(Q(K_2),\rho;f_{\theta},0))\mid\rho\in{\rm Hom}(Q(K_2),X)\}\\
& =\{\underbrace{(0),\ldots,(0)}_{102},\underbrace{(u-1),\ldots,(u-1)}_{48}\},
\end{align*}
where $\Z[\Z_4]$ is identified as $\Z[u^{\pm 1}]/(u^4-1)$. Therefore, these knots are not equivalent. 
\end{exam}
\begin{remark}
{\rm
Since $G(K_1)$ and $G(K_2)$ are isomorphic, the {\it twisted Alexander polynomial} \cite{lin2001repr,wada1994twisted} can not distinguish $K_1$ and $K_2$. Thus, the $f$-twisted Alexander matrix of knot quandles is a really stronger oriented knot invariant than the twisted Alexander polynomial of knots.
}
\end{remark}
\section{The deficiency of quandles}
\label{sect:deficiency qdle}
In this section, we study the deficiency of link quadles of oriented knots. At first, we define the deficiency of a quandle which have a finite presentation.
\begin{defi}
\label{defi:deficiency}
Let $P=\langle x_1,\ldots,x_n\mid r_1,\ldots,r_m\rangle$ be a finite presentation of a quandle $X$. The {\it deficiency of $P$}, which is denoted by ${\rm def}(P)$, is the integer ${\rm def}(P):=n-m$. The {\it deficiency of $X$}, which is denoted by ${\rm def}(X)$, is the maximal integer of ${\rm def}(P)$ for all finite presentation of $X$.
\end{defi}

\begin{exam}
\label{ex:defi}
Let $L$ be an oriented link and $D$ be a diagram of $L$. Let $n_a$ be the number of arcs of $D$ and $n_c$ be the number of crossings of $D$. By the Wirtinger presentation with respect to the diagram $D$, we have ${\rm def}(Q(L))\geq n_a-n_c$. 
\end{exam}
The deficiency of a quandle is an  analogy of the deficiency of a group. It is known that the deficiency of the knot group $G(K)$ is $1$. The aim of this section is to determine the deficiency of the link quandle of oriented knots.
\begin{theo}
\label{defi knot}
Let $K$ be an oriented knot. Then, we have
\[
{\rm def}(Q(K))=\begin{cases}
1\ \ (K{\rm :\ unknot}),\\
0\ \ ({\rm otherwise}).
\end{cases}
\]
\end{theo}
\begin{proof}
By Example~\ref{ex:defi}, we see that the deficiency of $Q(K)$ is bigger than or equal to $0$. Since Theorem~8.8.7 of \cite{kamada2017surface}, we can construct the presentation of $G(K)$ from the presentation of $Q(K)$. This correspondence does not change the cardinalities of generators and relators. Thus, the deficiency of $Q(K)$ is less than or equal to the deficiency of $G(K)$. Since the deficiency of $G(K)$ is $1$, it holds that the deficiency of $Q(K)$ is equal to $0$ or $1$.

If $K$ is the unknot, it is easily seen that the deficiency of $Q(K)$ is $1$. Hence, we consider that $K$ is non-trivial knot.

Assume that the deficiency of $Q(K)$ is $1$. There is a finite presentation $\langle x_1,\ldots,x_n\mid r_1,\ldots r_{n-1}\rangle$ of $Q(K)$. We fix this presentation.

Since $K$ is non-trivial knot, there exists a quandle $2$-cocycle $\lambda:Q(K)\times Q(K)\to\Z=\langle t\rangle$ such that $\Phi_{\lambda}(D,c_{id})=t$, where $D$ is a diagram of $K$ and $c_{id}:{\rm Arc}(D)\to Q(K)$ is the coloring of $D$ which corresponds to the identity map on $Q(K)$ (cf.~\cite{EISERMANN2003131}). By Theorem~\ref{quandle cocycle and ideal}, it holds that $E_0(A(Q(K),id;f_{\lambda},0))=(t-1)$, where $\Z[\Z]$ is identified as $\Z[t^{\pm 1}]$. On the other hand, the $f$-twisted Alexander matrix $A(Q(K),id;f_{\lambda},0)$ with respect to the presentation $\langle x_1,\ldots,x_n\mid r_1,\ldots r_{n-1}\rangle$ is an $n\times (n-1)$ matrix. By the definition of the $0$-th elementary ideal, we have $E_0(A(Q(K),id;f_{\lambda},0))=(0)$. This is a contradiction.
\end{proof}
\begin{remark}
{\rm
It is obvious that the deficiency of the knot quandle is a knot invariant. Theorem~\ref{defi knot} implies that the deficiency of the knot quandle detects the unknot.}
\end{remark}

\bibliographystyle{plain}
\bibliography{reference}
\end{document}